\documentclass[draft]{birkjour}
\usepackage[noadjust]{cite}
\usepackage{xcolor}
\RequirePackage[all]{xy}

\newtheorem{theorem}{Theorem}[section]
\newtheorem{corollary}[theorem]{Corollary}

\newtheorem{proposition}[theorem]{Proposition}
\theoremstyle{definition}
\newtheorem{defn}[theorem]{Definition}
\theoremstyle{remark}
\newtheorem{remark}[theorem]{Remark}

\numberwithin{equation}{section}

\newcommand{\BibTeX}{B\kern-0.1emi\kern-0.017emb\kern-0.15em\TeX}
\newcommand{\XYpic}{$\mathrm{X\kern-0.3em\raisebox{-0.18em}{Y}}$-$\mathrm{pic}\,$}

\newcommand{\cl}{C \kern -0.1em \ell}  %%Clifford algebra

\newcommand{\by}{\mathbf{y}}

\newcommand{\Rm}{\mathbb{R}^m}

\newcommand{\R}{\mathbb{R}}
\newcommand{\Clm}{\mathcal{C}l_m}

\newcommand{\Bm}{\mathbb{B}^{m}}
\newcommand{\Sm}{\mathbb{S}^{m-1}}
\newcommand{\Hk}{\mathcal{H}_k}
\newcommand{\Mk}{\mathcal{M}_k}
\newcommand{\Mkk}{\mathcal{M}_{k-1}}
\newcommand{\be}{\boldsymbol{e}}
\newcommand{\boy}{\boldsymbol{y}}

\newcommand{\bo}{\boldsymbol}
\newcommand{\bx}{\boldsymbol{x}}
\newcommand{\bu}{\boldsymbol{u}}
\newcommand{\bov}{\boldsymbol{v}}
\newcommand{\bt}{\boldsymbol{t}}
\newcommand{\bzeta}{\boldsymbol{\zeta}}
\newcommand{\bbeta}{\boldsymbol{\eta}}
\newcommand{\balpha}{\boldsymbol{\alpha}}
\newcommand{\bbbeta}{\boldsymbol{\beta}}
\newcommand{\bv}{\bigg\vert}
\newcommand{\baa}{\begin{align*}}
\newcommand{\eaa}{\end{align*}}

\allowdisplaybreaks

\newcommand{\ed}{\end{document}}
\begin{document}

%-------------------------------------------------------------------------
% editorial commands: to be inserted by the editorial office
%
%\firstpage{1} \volume{228} \Copyrightyear{2004} \DOI{003-0001}
%
%
%\seriesextra{Just an add-on}
%\seriesextraline{This is the Concrete Title of this Book\br H.E. R and S.T.C. W, Eds.}
%
% for journals:
%
%\firstpage{1}
%\issuenumber{1}
%\Volumeandyear{1 (2004)}
%\Copyrightyear{2004}
%\DOI{003-xxxx-y}
%\Signet
%\commby{inhouse}
%\submitted{March 14, 2003}
%\received{March 16, 2000}
%\revised{June 1, 2000}
%\accepted{July 22, 2000}
%
%
%
%---------------------------------------------------------------------------
%Insert here the title, affiliations and abstract:
%

\title[Properties and integral transforms in higher spin Clifford analysis]
 {Some properties and integral transforms in higher spin Clifford analysis}
%----------Author 1
\author[Chao Ding]{Chao Ding}
\address{Building H, Room 310,\\
Jiulong Road 111, Qingyuan Campus, \\
Center for Pure Mathematics, \\School of Mathematical Sciences,\\ Anhui University, Hefei, P.R. China}
\email{cding@ahu.edu.cn}

%\author[Zhenghua Xu]{Zhenghua Xu}
%
%\address{%
%School of Mathematics,\\
%Hefei University of Technology, Hefei, P.R. China}
%\email{zhxu@hfut.edu.cn}

%%%%

%

%\thanks{This file has been typeset with the option \texttt{draft} to illustrate that feature and its purpose.}
%----------Author 2
%\author[R.~Ab\l amowicz]{Rafa\l \ Ab\l amowicz}
%\author[]{Rafa\l \ Ab\l amowicz}
%\address{%
%E-i-C of AACA\\
%Sarasota, FL 34238}
%\email{rablamowicz.aaca@birkhauser-science.com}
%----------classification, keywords, date
\date{\today}
%----------additions
%\dedicatory{Last Revised:\\ \today}
%%% ----------------------------------------------------------------------
\begin{abstract}
Rarita-Schwinger equation plays an important role in theoretical physics. Bure\v s et al. generalized it to arbitrary spin $k/2$ in 2002 in the context of Clifford algebras. In this article, we introduce the mean value property, Cauchy's estimates, and Liouville's theorem for null solutions to Rarita-Schwinger operator in Euclidean spaces. Further, we investigate boundednesses to the Teodorescu transform and its derivatives. This gives rise to a Hodge decomposition of an $L^2$ spaces in terms of the kernel space of the Rarita-Schwinger operator and it also generalizes Bergman spaces in higher spin cases. 
\end{abstract}
\label{page:firstblob}
%%% ----------------------------------------------------------------------
\maketitle
%%% ----------------------------------------------------------------------
%\tableofcontents
\section{Introduction}~\par
In complex analysis, the Teodorescu transform plays an important role in the classical two dimensional Vekua theory. This transform is a two dimensional weak singular integral operator over a domain in the complex plane, which is the right inverse to the Cauchy-Riemann operator $\overline{\partial}$, and it is given by
\begin{align*}
T_{\Omega}f(z)=\int_{\Omega}\frac{f(\zeta)}{\zeta-z}d\zeta,
\end{align*}
where $z=u+iv\in\Omega$, and $\Omega$ is a domain in $\mathbb{C}$. In the classical book \cite{Vekua}, Vekua gave a comprehensive study on mapping properties and their applications in complex analysis. 
\par
In the last decades, higher dimensional generalizations and applications of the Teodorescu transform have been investigated by many researchers in the context of Clifford analysis. One of the first generalizations can be found in \cite{Sp1}. In \cite{Gur1,Gur2} G\"urlebeck and Spr\"o\ss ig developed an operator calculus in real Clifford algebras with special emphasis on Teodorescu transforms and their applications in solving certain boundary value problems. Spr\"o\ss ig \cite{Sp2} studied some elliptic boundary value problems with the Teodorescu transform. One of the important applications of the Teodorescu transform is to give the existence of solutions to the Beltrami equation, and there are many papers contributed to this topic. For instance, K\"ahler \cite{Kahler} generalized the Beltrami equation in cases of quaternions. G\"urlebeck and K\"ahler \cite{GK} investigated a hypercomplex generalization of the complex $\Pi$-operator and the solution of a hypercomplex Beltrami equation. More related work can be found, for instance, in \cite{BRAK,CK,GKS}.
\par
The higher spin theory in the context of Clifford analysis was firstly introduced by Bure\v s et al. \cite{Bures} in 2002. In \cite{Bures}, the authors introduced the generalized Rarita-Schwinger operator, which acts on functions taking values in $k$-homogeneous monogenic (null solutions to the Dirac operator) polynomials. It turns out that the Rarita-Schwinger operator is the natural generalization of the Dirac operator in the higher spin theory with many properties, such as Cauchy's Theorem, Cauchy integral formula, etc. are preserved. However, there are still some properties which have not been investigated in this context, for instance, mean value property, Cauchy's estimates, Liouville's theorem, etc. In this article, we will introduce the analogs of the mean value property, Cauchy's estimates and Liouville's theorem for the Rarita-Schwinger operator. We also study some properties of the Teodorescu transform. These results help us to obtain a Hodge decomposition related to the Rarita-Schwinger operator, which gives rise to generalized Bergman spaces in higher spin theory.
%Later, in 2016, Eelbode \cite{Eelbode} and De Bie \cite{DeBie} studied the generalizations of Laplacian in higher spin theory, which are named as the higher spin Laplace operators (also called bosonic Laplacians). Constructions, fundamental solutions to the higher spin Laplace operators and polynomial null solutions to the higher spin Laplace operators are studied there. 
\par
This article is organized as follows. Some definitions and notations of the Clifford algebras setting and Rarita-Schwinger operator are introduced in Section 2. Definitions of some function spaces in higher spin Clifford analysis are introduced in Section 3. Section 4 is devoted to properties for null solutions to the Rarita-Schwinger operator. Mapping properties for the Teodorescu transform in higher spin Clifford analysis are studied in Section 5. A Hodge decomposition for an $L^2$ in terms of the kernel space of the Rarita-Schwinger operator is given in Section 6. This gives rise to an analog of Bergman spaces in higher spin Clifford analysis ,which will be studied in an upcoming paper.
%%%%%%%%%%%%%%%%%%%%%%%%%%%%%
\section{Definitions and notations}
In this section, we review some definitions and preliminary results on Clifford analysis, for more details, we refer the readers to \cite{2,10,21}.
\par
Let $\R^m$ be the $m$-dimensional Euclidean space with a standard orthonormal basis $\{{\be}_1,\ldots,{\be}_m\}$. The real Clifford algebra $\Clm$ is generated by $\Rm$ with the following relationship
$${\be}_i{\be}_j+{\be}_j{\be}_i=-2\delta_{ij},$$
where $\delta_{ij}$ is the Kronecker delta function. Hence a Clifford number $x\in \Clm$ can be written as $x=\sum_{A}x_Ae_A$ with real coefficients and $A\subset\{1,\ldots,m\}$. This suggests that one can consider $\Clm$ as a vector space with dimension $2^m$. Therefore, a reasonable norm for a Clifford number $x=\sum_{A}x_Ae_A$ should be $|x|=(\sum_{A}x_A^2)^{\frac{1}{2}}$. If we denote $\Clm^{k}=\{x\in \Clm:\ x=\sum_{|A|=k}x_Ae_A\}$, where $|A|$ stands for the cardinality of the set $A$, then one can see that $\Clm=\displaystyle\bigoplus_{k=0}^m \Clm^k$. In particular, the $m$-dimensional Euclidean space $\R^{m}$ can be identified with $ \Clm^1$ as following
\begin{align*}
\R^{m}&\longrightarrow  \Clm^1,\\
(x_1,\ldots,x_m)&\longmapsto \bx=:x_1e_1+\cdots+x_me_m.
\end{align*}
The Dirac operator in $\mathbb{R}^m$ is defined to be $$D_{\bx}:=\sum_{i=1}^{m}e_i\partial_{x_i},$$ 
where $\partial_{x_i}$ stands for the partial derivative with respect to $x_i$. It is easy to  verify that $D_{\bx}^2=-\Delta_{\bx}$, where $\Delta_x$ is the Laplacian in $\mathbb{R}^m$. Therefore, we usually say Clifford analysis is a refinement of harmonic analysis.
\begin{defn}
 A $\Clm$-valued function $f(x)$ defined on a domain $U$ in $\Rm$ is \emph{left monogenic} if $D_{\bx}f(\bx)=0$. 
 \end{defn}
 Since Clifford multiplication is not commutative in general, there is a similar definition for right monogenic functions. Sometimes, we will consider the Dirac operator $D_{\bu}$ in a vector $\bu$ rather than $\bx$. In this article, we use monogenic to represent left monogenic unless specified.
\par
Let $\mathcal{M}_k$ denote the space of $\mathcal{C}l_m$-valued monogenic polynomials homogeneous of degree $k$. Note that if $h_k\in\Hk$, the space of $\mathcal{C}l_m$-valued harmonic polynomials homogeneous of degree $k$, then it is easy to see that $D_{\bu}h_k\in\mathcal{M}_{k-1}$, but $D_uup_{k-1}(\bu)=(-m-2k+2)p_{k-1}(\bu)$, so we have
$$\mathcal{H}_k=\mathcal{M}_k\oplus \bu\mathcal{M}_{k-1},\ h_k=p_k+\bu p_{k-1}.$$
This is a \emph{Fischer decomposition} of $\Hk$ \cite{DJR}. In particular, we denote the two projection maps by 
\begin{align*}
&P_k: \mathcal{H}_k\longrightarrow \mathcal{M}_k,\\
&Q_k: \mathcal{H}_k\longrightarrow \bu\mathcal{M}_{k-1}.
\end{align*}
Suppose $\Omega$ is a domain in $\mathbb{R}^m$, we consider a differentiable function $f: \Omega\times \mathbb{R}^m\longrightarrow \mathcal{C}l_m$
such that, for each $\bx\in \Omega$, $f(\bx,\bu)$ is a left monogenic polynomial homogeneous of degree $k$ in $\bu$. Then, the first order conformally invariant differential operator in higher spin theory, named as the Rarita-Schwinger operator \cite{Bures,DJR}, is defined by 
\begin{eqnarray*}
R_kf(\bx,\bu):=P_kD_{\bx}f(\bx,\bu)=\bigg(1+\frac{\bu D_{\bu}}{m+2k-2}\bigg)D_{\bx}f(\bx,\bu).
\end{eqnarray*}
%If a function $f\in C^1(\Omega\times\Bm,\Mk)$ satisfies $R_kf=0$, we call it a \emph{fermionic function}. 
The identity of $\text{End}(\Mk)$ can be represented by the reproducing kernel $Z_k(u,v)$ for the inner spherical monogenics of degree $k$. This so called zonal spherical monogenic satisfies
\begin{eqnarray*}
f(\bu)=\int\displaylimits_{\Sm}{Z_k(\bu,\bov)}f(\bov)dS(\bov),\ \text{for\ all}\ f\in\Mk.
\end{eqnarray*}
Then, the fundamental solution for $R_k$ is given by
\begin{eqnarray}\label{fund}
E_{k}(\bx,\bu,\bov)=\frac{1}{c_{m,k}}\frac{\bx}{|\bx|^m}Z_k\bigg(\frac{\bx\bu\bx}{|\bx|^2},\bov\bigg),
\end{eqnarray}
where $c_{m,k}=\displaystyle\frac{(m-2)\omega_{m-1}}{m+2k-2}$ and $\omega_{m-1}$ is the area of $(m-1)$-dimensional unit sphere.
\par
\begin{defn}
A function $f\in C^1(\Omega\times\Bm,\Mk)$ is called \emph{fermionic}, if it satisfies $R_kf(\bx,\bu)=0$ for all $(\bx,\bu)\in\Omega\times\Bm$.
\end{defn}
The space of fermionic functions can be considered as an analog of the space of holomorphic functions in higher spin Clifford analysis. It has been shown that fermionic functions have properties similar to those of holomorphic functions, such as Cauchy's theorem, Cauchy integral formula, etc., see \cite{Bures,DJR}. In this article, we will introduce more properties from complex analysis preserved by fermionic functions in Section 4. 
%%%%%%%%%%%%%%%%%%%%%%%%%
\section{Function spaces in higher spin Clifford analysis}
In this section, we introduce generalizations of some classical function spaces in higher spin theory. Let $\Omega\subset\R^m$ be a domain, the norm of the space $L^p(\Omega\times \Bm,\Clm)$ is given by
\begin{align*}
	||f||_{L^p}:=\bigg(\int_{\Omega}\int_{\Bm}|f(\bx,\bu)|^pdS(\bu)d\bx\bigg)^{\frac{1}{p}}.
\end{align*}
We introduce the Sobolev space $W^{s,t}_p(\Omega\times \Bm,\Clm)$ as the subset of functions $f$ in $L^p(\Omega\times \Bm,\Clm)$ such that $f$ and its derivatives up to order-$(s,t)$ have a finite $L^p$ norm. Let $\bo{\alpha}=(\alpha_1,\ldots,\alpha_m),\bo{\beta}=(\beta_1,\ldots,\beta_m)$ be multi-indices, $|\bo{\alpha}|=\sum_{j=1}^m\alpha_j,\ |\bo{\beta}|=\sum_{j=1}^m\beta_j$, and $\partial_{\bx}^{\bo{\alpha}}:=(\partial^{\alpha_1}_{x_1},\ldots,\partial_{x_m}^{\alpha_m})$, $\partial_{\bu}^{\bo{\beta}}:=(\partial^{\beta_1}_{u_1},\ldots,\partial_{u_m}^{\beta_m})$. The norm for $W^{s,t}_p(\Omega\times \Bm,\Clm)$ with $p>1$ is given by
\begin{align*}
||f||_{W^{s,t}_p(\Omega\times \Bm,\Clm)}:=\bigg[\sum_{|\bo{\alpha}|=0}^s\sum_{|\bo{\beta}|=0}^t||\partial_{\bx}^{\bo{\alpha}}\partial_{\bu}^{\bo{\beta}}f||^p_{L^p(\Omega\times\Bm,\Clm)}\bigg]^{\frac{1}{p}}.
\end{align*}
Now, we say that a function $f(\bx,\bu)\in L^p(\Omega\times \Bm,\Mk)$ if for each fixed $\bx\in\Omega$, $f(\bx,\bu)\in\Mk$ with respect to $\bu$ and the $L^p$ norm of $f$ is finite. 
In later sections, for Sobolev spaces $W^{s,t}_p(\Omega\times \Bm,V)$ with $V=\Mk\ \text{or}\ \Hk$, since it is a space of $k$-homogeneous polynomials in the variable $\bu$, we can omit the regularity with respect to $\bu$ to $W^{s}_p(\Omega\times \Bm,V)$ instead.
\par
 Now, we claim that
\begin{theorem}
	$L^p(\Omega\times \Bm,\Mk)$ is a closed subspace of $L^p(\Omega\times \Bm,\Clm)$.
	\end{theorem}
\begin{proof}
	Let $\{f_n\}_{n=1}^{\infty}$ be a Cauchy sequence in $L^p(\Omega\times \Bm,\Mk)$, which converges to $f\in L^p(\Omega\times \Bm,\Clm)$. This means that 
	\begin{align*}
		\lim_{n\rightarrow\infty}\int_{\Omega}\int_{\Sm}|f_n(\bx,\bu)-f(\bx,\bu)|^pdS(\bu)d\bx=0,
	\end{align*}
	and this gives rise to that for almost every $(\bx,\bu)\in\Omega\times \Bm$, we have $$\lim_{n\rightarrow\infty}f_n(\bx,\bu)=f(\bx,\bu).$$
	Further, since for each fixed $\bx\in\Omega$, $f_n(\bx,\bu)\in \Mk$ and $\Mk$ is a finite dimensional vector space, let $\{\varphi_j(\bu)\}_{j=1}^s$ be an orthonormal basis of $\Mk$. Then, for each fixed $\bx\in\Omega$, we can rewrite 
	\begin{align*}
	f_n(\bx,\bu)=\sum_{j=1}^s\varphi_{j}(\bu)a_{n,j}(\bx),
	\end{align*}
	where $a_{n,j}(\bx)$ are all Clifford-valued numbers depending on $\bx$. Apparently, the convergence of $\{f_n(\bx,\bu)\}_{n=1}^{\infty}$ is equivalent to the convergence of the sequence $\{a_{n,j}(\bx)\}_{n=1}^{\infty}$ for all $j=1,2,\cdots,s$, we assume that $\lim_{n\rightarrow\infty}a_{n,j}(\bx)=a_j(\bx)$ for $j=1,2,\cdots,s$. Therefore, we must have 
		\begin{align*}
	f(\bx,\bu)=\sum_{j=1}^s\varphi_{j}(\bu)a_{j}(\bx),
	\end{align*}
	which implies that $f\in \Mk$, i.e., $f\in L^p(\Omega\times \Bm,\Mk)$.
\end{proof}
\begin{remark}
One might notice that the integral with respect to $\bu$ in the $L^p$ norm used above is over $\Sm$ instead of $\Bm$. This is because they are equivalent since the functions considered are $k$-homogeneous with respect to $u$.
\end{remark}
%%%%%%%%%%%%%%%%%%%%%%%%%%%%%%%%%%
\section{Properties for fermionic functions}
In this section, we investigate mean value property, Cauchy's estimates and Liouville's theorem for fermionic functions. This once again shows that fermio-nic functions behave similarly to monogenic functions.
\par
Firstly, we introduce a mean value property for fermionic functions as follows.
\begin{theorem}[Mean Value Property]\label{MVP1}~\\
Let $f\in C^1(\Omega\times\Bm,\Mk)\cap C(\overline{\Omega}\times\overline{\Bm},\Mk)$ and $R_kf=0$, then for all $(\boy,\bu)\in\Omega\times\Bm$, we have
\begin{align*}
f(\boy,\bu)=&\frac{1}{\omega_{m-1}c_k}\int_{|\bt|=1}P_{k,\bu}f(\boy+r\bt,\bt\bu\tilde{\bt})d\bt\\
=&\frac{1}{c_kV(B(\boy,r))}\int_{B(\boy,r)}P_{k,\bu}f\bigg(\bx,\frac{(\bx-\boy)\bu\widetilde{(\bx-\boy)}}{|\bx-\boy|^2}\bigg)d\bx.
\end{align*}
\end{theorem}
\begin{proof}
Recall that for $f\in\ker R_k$, the Cauchy integral formula tells us that
\begin{align*}
f(\boy,\bu)=&\int_{\partial B(\boy,r)}\int_{\Sm}E_k(\bx-\boy,\bu,\bov)n(\bx)f(\bx,\bov)dS(\bov)d\sigma(\bx)\\
=&\int_{|\bt|=1}\int_{\Sm}\frac{1}{\omega_{m-1}c_k}Z_k(\bu,\tilde{\bt}\bov\bt)\frac{\bt}{r^{m-1}}\bt f(\boy+r\bt,\bov)dS(\bov)r^{m-1}d\bt\\
=& \frac{1}{\omega_{m-1}c_k}\int_{|\bt|=1}\int_{\Sm}\tilde{\bt}Z_k(\bt\bu\tilde{\bt},\bov)\big(P_{k,\bov}\bt f(\boy+r\bt,\bov)\big)dS(\bov)d\bt\\
=&\frac{1}{\omega_{m-1}c_k}\int_{|\bt|=1}\tilde{\bt}\big(P_{k,\bov}\bt f(\boy+r\bt,\bov)\big)\bv_{\bov=\bt\bu\tilde{\bt}}d\bt\\
=&\frac{1}{\omega_{m-1}c_k}\int_{|\bt|=1}P_{k,\bu} f(\boy+r\bt,\bt\bu\tilde{\bt})d\bt,
\end{align*}
where the last equation comes from the fact that $\tilde{\bt}P_{k,\bov}\bt=P_{k,\bu}$, where $\bov=\bt\bu\tilde{\bt}$. Now, we calculate the integral over $B(\boy,r)$ as follows.
\begin{align*}
&\frac{1}{c_kV(B(\boy,r))}\int_{B(\boy,r)}P_{k,\bu}f\bigg(\bx,\frac{(\bx-\boy)\bu\widetilde{(\bx-\boy)}}{|\bx-\boy|^2}\bigg)d\bx\\
=&\frac{1}{c_kV(B(\boy,r))}\int_0^r\int_{|\bt|=1}P_{k,\bu}f(\boy+\epsilon\bt,\bt\bu\tilde{\bt})\epsilon^{m-1}dS(\bt)d\epsilon\\
=&\frac{1}{c_kV(B(\boy,r))}\int_0^r\epsilon^{m-1}\omega_{m -1}c_kf(\boy,\bu)d\epsilon\\
=&\frac{\omega_{m-1}r^m}{mV(B(\boy,r))}f(\boy,\bu)=f(\boy,\bu),
\end{align*}
where the second equation above relies on the mean value property over $|\bt|=1$, and this completes the proof.
\end{proof}
\begin{remark}
The reason that we name the theorem above as a mean value property is the following: in \cite{DNR1}, we introduced a mean value property for bosonic Laplacians $\mathcal{D}_k: C^{2}(\Omega\times\Bm,\Hk)\longrightarrow C^{2}(\Omega\times\Bm,\Hk)$ as 
\begin{align}\label{MVP}
f(\bx,\bov)=\int_{\Sm}f(\bx+r\bzeta,\bzeta\bov\bzeta)dS(\bzeta),
\end{align}
for all functions $f\in C^{2}(\Omega\times\Bm,\Hk)$ satisfying $\mathcal{D}_kf=0$. In \cite[Proposition 1]{DR}, a decomposition of $\mathcal{D}_k$ in terms of the Rarita-Schwinger type operators was given. More importantly, for functions $f\in C^2(\Omega\times\Bm,\Mk)$, the decomposition gives us that $\mathcal{D}_k=A_kR_k$, where $A_k$ is a first order differential operator of $\bx$ given in \cite[Proposition 2]{DR}. This implies that for a function $f\in C^2(\Omega\times\Bm,\Mk)$ satisfying $R_kf=0$, we immediately have $\mathcal{D}_kf=0$. Therefore, with \eqref{MVP} and $P_{k,\bov}f=f$ for  $f\in C^2(\Omega\times\Bm,\Mk)$, we have
\begin{align*}
f(\bx,\bov)=&P_kf(\bx,\bov)=P_{k,\bov}\int_{\Sm}f(\bx+r\bzeta,\bzeta\bov\bzeta)dS(\bzeta)\\
=&\int_{\Sm}P_{k,\bov}f(\bx+r\bzeta,\bzeta\bov\bzeta)dS(\bzeta),
\end{align*}
which gives the mean value property in Theorem \ref{MVP1}.
\end{remark}
Let $\balpha=(\alpha_1,\cdots,\alpha_m)$ and $\bbbeta=(\beta_1,\cdots,\beta_m)$ be multi-indices, where $\alpha_j,\beta_j, j=1,\cdots,m$ are all non-negative integers. It is easy to see that for each $\bx\in\partial\Omega,\bov\in\Sm$, the fundamental solution $E(\bx-\boy,\bu,\bov)$ given in \eqref{fund} is infinitely differentiable in $\Omega\times\Bm$ with respect to $\boy,\bu$.Therefore, $D_{\boy}^{\balpha}D_{\bu}^{\bbbeta}E(\bx-\boy,\bu,\bov)$ are integrable over $\Omega\times\Bm$.
\par
Recall that for $f\in C^1(\Omega\times\Bm,\Mk)\cap C(\overline{\Omega}\times\overline{\Bm},\Mk)$ and $R_kf=0$, the Cauchy integral formula \cite[Theorem 8]{DJR} gives us that
\begin{align*}
D_{\boy}^{\balpha}D_{\bu}^{\bbbeta}f(\boy,\bu)=&\int_{\partial \Omega}\int_{\Sm}\big(D_{\boy}^{\balpha}D_{\bu}^{\bbbeta}E_k(\bx-\boy,\bu,\bov)\big)n(\bx)f(\bx,\bov)dS(\bov)d\sigma(\bx),
\end{align*}
for any $(\boy,\bu)\in\Omega\times\Bm$. Differentiating under the integral sign, one can see that $f(\boy,\bu)$ is infinitely differentiable with respect to $(\boy,\bu)\in\Omega\times\Bm$. Now, we can easily derive an analog of Cauchy's estimates as follows.
\begin{theorem}[Cauchy's estimates with $L^{\infty}$ norm]\label{CE}
Let $\balpha,\bbbeta$ be multi-indices. Assume $f\in C^1(\Omega\times\Bm,\Mk)$ and $R_kf=0$. Further, let $r_1,r_2>0$ be sufficiently small such that $\overline{B(\boy,r_1)}\subset\Omega$ and $\overline{B(\bu,r_2)}\subset\Sm$. Then, there exists a constant $c_{\balpha,\bbbeta,m,k}$ such that 
\begin{align*}
|D_{\boy}^{\balpha}D_{\bu}^{\bbbeta}f(\boy,\bu)|\leq\frac{c_{\balpha,\bbbeta,m,k}||f||_{L^{\infty}(B(\boy,r_1)\times B(\bu,r_2),\Mk))}}{r_1^{|\balpha|}r_2^{|\bbbeta|}}.
\end{align*}
\end{theorem}
\begin{proof}
Firstly, we consider the case $\Omega=\Bm$, and let $||f||_{L^{\infty}(\Bm\times\Bm,\Mk)}=M'$, then for $\bu\in\Bm$, we have
\begin{align*}
&|D_{\boy}^{\balpha}D_{\bu}^{\bbbeta}f(\boy,\bu)|\\
=&\bigg\vert\int_{\Sm}\int_{\Sm}\big(D_{\boy}^{\balpha}D_{\bu}^{\bbbeta}E_k(\bx-\boy,\bu,\bov)\big)n(\bx)f(\bx,\bov)dS(\bov)d\sigma(\bx)\bigg\vert\\
\leq&\int_{\Sm}\int_{\Sm}\big\vert D_{\boy}^{\balpha}D_{\bu}^{\bbbeta}E_k(\bx-\boy,\bu,\bov)\big\vert dS(\bov)d\sigma(\bx)\cdot M'.
\end{align*}
Recall that
\begin{align*}
E_{k}(\bx-\boy,\bu,\bov)=\frac{1}{c_{m,k}}\frac{\bx-\boy}{|\bx-\boy|^m}Z_k\bigg(\frac{(\bx-\boy)\bu(\bx-\boy)}{|\bx-\boy|^2},\bov\bigg),
\end{align*}
and $Z_k(\bu,\bov)$ is a homogeneous polynomial of degree $k$ with respect to $\bu,\bov$. This implies that $E_{k}(\bx-\boy,\bu,\bov)$ has no singular points for $\bx,\bov\in\Sm$, and  
\begin{align*}
|D_{\boy}^{\balpha}D_{\bu}^{\bbbeta}E_k(\bx-\boy,\bu,\bov)|
\end{align*}
is bounded for fixed $\boy,\bu\in\Bm$ and $\bx,\bov\in\Sm$. Therefore, there exists a constant $c_{\balpha,\bbbeta,m,k}$, which only depends on $\balpha,\bbbeta,m$ and $k$, such that
\begin{align*}
\int_{\Sm}\int_{\Sm}\big\vert D_{\boy}^{\balpha}D_{\bu}^{\bbbeta}E_k(\bx-\boy,\bu,\bov)\big\vert dS(\bov)d\sigma(\bx)\leq c_{\balpha,\bbbeta,m,k}.
\end{align*}
Hence, we have that 
\begin{align*}
|D_{\boy}^{\balpha}D_{\bu}^{\bbbeta}f(\boy,\bu)|\leq c_{\balpha,\bbbeta,m,k} M',
\end{align*}
 for all $(\boy,\bu)\in\Bm\times\Bm$. Now, we consider $(\boy,\bu)\in B(\boy,r_1)\times B(\bu,r_2)$. We assume that $R_kf=0$ in $\Omega\times\Bm$ and $||f||_{L^{\infty}(B(\boy,r_1)\times B(\bu,r_2),\Mk))}=M$. Then, we apply the argument above to the function $f(\boy+r_1\bzeta,\bu+r_2\bbeta)$ with respect to $\bzeta,\bbeta\in\Bm$, we obtain that
\begin{align*}
|D_{\boy}^{\balpha}D_{\bu}^{\bbbeta}f(\boy,\bu)|\leq\frac{c'_{\balpha,\bbbeta,m,k}M}{r_1^{|\balpha|}r_2^{|\bbbeta|}},
\end{align*}	
which completes the proof.
\end{proof}
One can also obtain Cauchy's estimates with an $L^1$ norm as follows.
\begin{theorem}[Cauchy's estimates with $L^1$ norm]
Let $\balpha,\bbbeta$ be multi-indices. Assume $f\in C^1(\Omega\times\Bm,\Mk)$ and $R_kf=0$ in $\Omega\times\Bm$. Then, there exists a constant $c_{\balpha,\bbbeta,m,k}$ such that for any $\overline{B(\boy,r_1)}\subset\Omega$ and $\bu\in\Bm$, we have
\begin{align*}
|D_{\bu}^{\bbbeta}D_{\boy}^{\balpha}f(\boy,\bu)|\leq \frac{c_{\balpha,\bbbeta,m,k}||f||_{L^1(B(\boy,r_1)\times B(0,1-\frac{r_2}{4}),\Mk)}}{r_1^{m+|\balpha|}r_2^{m+|\bbbeta|}},
\end{align*}
where $r_2$ stands for the distance from $\bu\in\Bm$ to the boundary $\Sm$ and 
\begin{align*}
||f||_{L^1(B(\boy,r_1)\times B(0,1-\frac{r_2}{4}),\Mk)}:=\int_{B(\boy,r_1)}\int_{B(0,1-\frac{r_2}{4})}|f(\boy,\bu)|d\bu d\boy.
\end{align*}
\end{theorem}
\begin{proof}
Let $M=||f||_{L^{\infty}(B(\boy,\frac{r_1}{2})\times B(\bu,\frac{r_2}{2}),\Mk)}<\infty$. Then, we apply the argument in the previous theorem above to $f(\boy+\frac{r_1}{2}\bzeta,\bu+\frac{r_2}{2}\bbeta)$ with respect to $\bzeta,\bbeta\in\Bm$ to have
\begin{align}\label{Meqn0}
|D_{\bu}^{\balpha}D_{\boy}^{\bbbeta}f(\boy,\bu)|\leq \frac{c_{\balpha,\bbbeta,m,k}M}{r_1^{|\balpha|}r_2^{|\bbbeta|}}.
\end{align}
Now, we assume that $f$ obtains its maximum value $M$ over $B(\boy,\frac{r_1}{2})\times B(\bu,\frac{r_2}{2})$ at the point $(\boy_1,\bu_1)$. One can easily see that
\begin{align*}
B\bigg(\boy_1,\frac{r_1}{2}\bigg)\subset B(\boy,r_1),\ B\bigg(\bu_1,\frac{r_2}{2}\bigg)\subset B\bigg(\bu,\frac{3r_2}{4}\bigg).
\end{align*}
Therefore, with the mean value property in Theorem \ref{MVP1}, we have
\begin{align}\label{Meqn}
M=&|f(\boy_1,\bu_1)|\nonumber\\
=&\frac{m+2k-2}{(m-2)V(B(\boy_1,\frac{r_1}{2}))}\bv\int_{B(\boy_1,\frac{r_1}{2})}P_{k,\bu_1}f\bigg(\boy,\frac{(\boy-\boy_1)\bu_1(\boy-\boy_1)}{|\boy-\boy_1|^2}\bigg)d\boy\bv\nonumber\\
\leq&\frac{m+2k-2}{(m-2)V(B(\boy_1,\frac{r_1}{2}))}\int_{B(\boy_1,\frac{r_1}{2})}\bv f\bigg(\boy,\frac{(\boy-\boy_1)\bu_1(\boy-\boy_1)}{|\boy-\boy_1|^2}\bigg)\bv d\boy.
\end{align}
Now, we denote $\bu_2=\frac{(\boy-\boy_1)\bu_1(\boy-\boy_1)}{|\boy-\boy_1|^2}$, which is a reflection of $\bu_1$ in the direction of $\boy-\boy_1$. This implies that $f(\boy,\bu_2)$ is harmonic with respect to $\bu_2$. Further, we notice that $B(\bu_2,\frac{r_2}{2})\subset B(0,1-\frac{r_2}{4})$, hence, with the mean value property of harmonic functions, we have 
\begin{align*}
|f(\boy,\bu_2)|=&\frac{1}{V(B(\bu_2,\frac{r_2}{2}))}\bv\int_{B(\bu_2,\frac{r_2}{2})}f(\boy,\bu)d\bu\bv\\
\leq& \frac{c_m}{r_2^m}\int_{B(0,1-\frac{r_2}{4})}|f(\boy,\bu)|d\bu,
\end{align*}
where $c_m$ is a constant only depending on $m$. 
Therefore, with \eqref{Meqn}, we have
\begin{align*}
M=&|f(\boy_1,\bu_1)|\leq \frac{c'}{r_2^mV(B(\boy_1,\frac{r_1}{2}))}\int_{B(\boy_1,\frac{r_1}{2})}\int_{B(0,1-\frac{r_2}{4})}|f(\boy,\bu)|d\boy d\bu\\
=&c''r_1^{-m}r_2^{-m}||f||_{L^1(B(\boy,r_1)\times B(0,1-\frac{r_2}{4}),\Mk)},
\end{align*}
where the last equation is because of the fact that $B(\boy_1,\frac{r_1}{2})\subset B(\boy,r_1)$ and $c',c''$ are positive constants only depending on $m$ and $k$. Combining with equation \eqref{Meqn0}, we eventually have
\begin{align*}
|D_{\bu}^{\bbbeta}D_{\boy}^{\balpha}f(\boy,\bu)|\leq \frac{c_{\balpha,\bbbeta,m,k}||f||_{L^1(B(\boy,r_1)\times B(0,1-\frac{r_2}{4}),\Mk)}}{r_1^{m+|\balpha|}r_2^{m+|\bbbeta|}},
\end{align*}
which completes the proof.
\end{proof}
%%%%%%%%%%%%%%%%%%%%%%%%%%%%%%%%%%%%%%%%
Now, we can obtain an analog of Liouville's theorem for fermionic functions as follows.
\begin{theorem}[Liouville's theorem] Let $f\in C^1(\Rm\times\Bm,\Mk)\cap L^{\infty}(\Rm\times\Bm,\Mk)$, and $R_kf(\bx,\bu)=0$ for all $(\bx,\bu)\in\Rm\times\Bm$. Then, we must have $f=f(\bu)\in \Mk(\bu)$.
\end{theorem}
\begin{proof}
Let $B(\bx,r)\subset\Rm$ be an arbitrary ball and $M=L^{\infty}(\Rm\times\Bm,\Mk)$. Theorem \ref{CE} tells us that
\begin{align*}
\bigg\vert\frac{\partial}{\partial x_i}f(\bx,\bu)\bigg\vert\leq \frac{c_{m,k}M}{r},
\end{align*}
for $i=1,\cdots,m$. Letting $r\rightarrow \infty$, we have $|\frac{\partial}{\partial x_i}f(\bx,\bu)|=0$ for $i=1,\cdots,m$, which implies that $f$ is independent to $x_i, i=1,\ldots,m$. In other words, we must have $f=f(\bu)\in \Mk(\bu)$.
\end{proof}
An immediate consequence of Liouville's theorem above is the following.
\begin{corollary}
Let $l\geq 1$ be an integer. Suppose $f\in C^1(\Rm\times\Bm,\Mk)$, $R_kf=0$ in $\Rm\times\Bm$ and
\begin{align*}
||f||_{L^1(B(0,R)\times\Bm,\Mk)}=o(R^{m+l}),\ \text{as}\ r\rightarrow\infty.
\end{align*}
Then, $f$ is a polynomial of $\bx$ with degree less than $l$.
\end{corollary}
\begin{proof}
Let $(\boy_0,\bu)$ be an arbitrary point in $\Rm\times\Bm$ and $B(\bu,r_2)\in\Bm$. Now, we denote $|\boy_0|=r$, then one can see that $B(\boy_0,R)\subset B(0,R+2r)$. In accordance to the Cauchy's estimates, for any multi-index $\balpha$, we have
\begin{align*}
&|D_{\boy}^{\balpha}f(\boy_0,\bu)|\leq \frac{c_{\balpha,\bbbeta,m,k}||f||_{L^1(B(\boy_0,R)\times\Bm,\Mk)}}{R^{m+|\balpha|}r_2^{m}}\\
\leq& \frac{c_{\balpha,\bbbeta,m,k}||f||_{L^1(B(0,R+2r)\times\Bm,\Mk)}}{R^{m+|\balpha|-1}}= \frac{c_{\balpha,\bbbeta,m,k}o((R+2r)^{m+l})}{R^{m+|\balpha|}r_2^{m}},\ \text{as}\ R\rightarrow \infty.
\end{align*}
If we let $|\balpha|=l,\ R\rightarrow\infty$, and since $r_2$ is fixed, we can see that $|D_{\boy}^{\balpha}f(\boy_0,\bu)|=0$ for any $\balpha=l$ and $(\boy,\bu)\in\Rm\times\Bm$. This immediately tells us that $f$ is a polynomial of $\boy$ with degree less than $l$.
\end{proof}
%%%%%%%%%%%%%%%%%%%%%%%%%%%%%%%%%%
\section{Teodorescu transform and its properties}
In \cite[Theorem 7]{DJR}, the authors introduced the Cauchy-Pompeiu formula in higher spin theory, and they also introduced the Teodorescu transform as
\begin{align*}
T_kf(\boy,\bu)=-\int_{\Omega}\int_{\Sm}E_k(\bx-\boy,\bu,\bov)f(\bx,\bov)dS(\bov)d\bx,
\end{align*}
for $f\in C^1(\Omega\times \Bm,\Mk)$. Further, they also pointed out that $T_k$ is a right inverse to $R_k$ for $f\in C_0^{\infty}(\R^m\times \Bm,\Mk)$.
\par
It is worth pointing out that there is a type error in \cite[Definition 2, Theorem 10]{DJR}, and that is $\bu$ and $\bov$ should be switched there. One can check that at the very end of the proof of Theorem $10$, the authors used the fact that $Z_k(\bu,\bov)$ is the reproducing kernel of $k$-homogeneous spherical monogenic polynomials to obtain $\phi$ there, but one can easily see that $\bu$ and $\bov$ should be switched based on the equation given above Lemma $5$ in page $11$ there. That is the reason that our $T_k$ here is different from the one given in \cite{DJR}. Further, for a bounded domain $\Omega\in\R^m$, the condition for the function space in \cite[Theorem 10]{DJR} can be $C^1(\Omega\times \Bm,\Mk)$ with exactly the same argument there. For convenience, we restate that $T_k$ is a right inverse of $R_k$ correctly as follows.
\begin{theorem}
Let $\Omega\subset\R^m$ be a domain, and $f\in C^1(\Omega\times \Bm,\Mk)$. Then, we have
\begin{align*}
R_kT_kf(\boy,\bu)=-R_k\int_{\Omega}\int_{\Sm}E_k(\bx-\boy,\bu,\bov)f(\bx,\bov)dS(\bov)d\bx=f(\boy,\bu).
\end{align*}
\end{theorem} 
Now, we claim that $T_k$ is a bounded operator on $ L^p(\Omega\times \Bm,\Mk)$ with $p>1$. More specifically,
\begin{theorem}\label{Tkb}
Let $\Omega\subset\R^m$ be a domain, and $f\in L^p(\Omega\times \Bm,\Mk)$ with $p>1$. Then, we have
\begin{align*}
||T_kf||_{L^p(\Omega\times \Bm,\Mk)}\leq C||f||_{L^p(\Omega\times \Bm,\Mk)}.
\end{align*}
\end{theorem}
\begin{proof}
To prove the statement, we only need to notice that the reproducing kernel $Z_k(\bu,\bov)$ is bounded for all $\bu,\bov\in \Bm$, which gives us that
\begin{align*}
&||T_kf(\boy,\bu)||_{L^p(\Omega\times \Bm,\Mk)}\\
=&\bigg(\int_{\Omega}\int_{\Sm}\bv\int_{\Omega}\int_{\Sm}E_k(\bx-\boy,\bu,\bov)f(\bx,\bov)dS(\bov)d\bx\bv^pdS(\bu)d\boy\bigg)^{\frac{1}{p}}\\
\leq&C\bigg(\int_{\Omega}\int_{\Sm}\bv\int_{\Omega}\int_{\Sm}\frac{\bx-\boy}{|\bx-\boy|^m}f(\bx,\bov)dS(\bov)d\bx\bv^pdS(\bu)d\boy\bigg)^{\frac{1}{p}}\\
\leq&C\bigg(\int_{\Sm}\bigg[\int_{\Omega}\bv\int_{\Omega}\frac{\bx-\boy}{|\bx-\boy|^m}f(\bx,\bov)d\bx\bv^pd\boy\bigg] dS(\bov)\bigg)^{\frac{1}{p}}.
\end{align*}
Note that the term 
\begin{align*}
\int_{\Omega}\frac{\bx-\boy}{|\bx-\boy|^m}f(\bx,\bov)d\bx
\end{align*}
in the bracket above is actually the Teodoresu transform in the classical case with a parameter $\bov$ , see \cite[pp.151]{23}, and it is a bounded operator on $L^p(\Omega)$ with $p>1$. Therefore, the last integral above becomes
\begin{align*}
\leq&C\bigg(\int_{\Sm}\int_{\Omega}|f(\bx,\bov)|^pd\bx dS(\bov)\bigg)^{\frac{1}{p}}=C||f||_{L^p(\Omega\times \Bm,\Mk)}.
\end{align*}
\end{proof}
To establish a mapping property of $T_k$ on Sobolev spaces, we need to calculate the derivatives of $T_k$ first, then we should evaluate the mapping property of its derivatives, which is stated as follows.
\begin{theorem}
Let $\Omega\subset\R^m$ be a domain, and $f\in C^1(\Omega\times \Bm,\Mk)$,
we have
\begin{align*}
&\frac{\partial}{\partial y_i}T_kf(\boy,\bu)=\int_{\Omega}\int_{\Sm}\frac{\partial}{\partial y_i}E_k(\bx-\boy,\bu,\bov)f(\bx,\bov)dS(\bov)d\bx\\
&+c_{m,k}^{-1}\int_{|\bt|=1}t_i\bt f(\boy,\bt\bu\bt)d\sigma(\bt)\\
&-c_{m,k}^{-1}\int_{|\bt|=1}\frac{2\bt}{2-m}\bigg[u_i\langle\bt,D_{\bbeta}\rangle+\langle\bt,\bu\rangle\frac{\partial}{\partial\eta_i}+2t_i\langle\bt,\bu\rangle\langle\bt,D_{\bbeta}\rangle\bigg]f(\boy,\bt\bu\bt)d\sigma(\bt),
\end{align*}
where $\bbeta=\bt\bu\bt$. Further, we claim that
\begin{align*}
T_k:\ L^p(\Omega\times \Bm,\Mk)\longrightarrow W^{1}_p(\Omega\times\Bm,\Mk).
\end{align*}
\end{theorem}
\begin{proof}
 Let $\epsilon>0$ be sufficiently small such that $B(\boy,\epsilon)\subset\Omega$. We denote $\Omega_{\epsilon}=\Omega\backslash B(\boy,\epsilon)$, and $\bo{\eta}=\displaystyle\frac{(\bx-\boy)\bu(\bx-\boy)}{|\bx-\boy|^2}$. Then, we have
\begin{align*}
&c_{m,k}T_kf(\boy,\bu)\\
=&-\lim_{{\epsilon}\rightarrow 0}\int_{\Omega_{\epsilon}}\int_{\Sm}\frac{\bx-\boy}{|\bx-\boy|^m}Z_k\bigg(\frac{(\bx-\boy)\bu(\bx-\boy)}{|\bx-\boy|^2},\bov\bigg)f(\bx,\bov)dS(\bov)d\bx\\
=&-\lim_{{\epsilon}\rightarrow 0}\int_{\Omega_{\epsilon}}\frac{\bx-\boy}{|\bx-\boy|^m}f\bigg(\bx,\frac{(\bx-\boy)\bu(\bx-\boy)}{|\bx-\boy|^2}\bigg)d\bx\\
=&-\lim_{{\epsilon}\rightarrow 0}\int_{\Omega_{\epsilon}}\bigg(\frac{|\bx-\boy|^{2-m}}{2-m}D_{\bx}\bigg)f(\bx,\bbeta)d\bx\\
=&\lim_{{\epsilon}\rightarrow 0}\int_{\Omega_{\epsilon}}\frac{|\bx-\boy|^{2-m}}{2-m}\big(D_{\bx}f(\bx,\bbeta)\big)d\bx-\int_{\partial\Omega_{\epsilon}}\frac{|\bx-\boy|^{2-m}}{2-m}n(\bx)f(\bx,\bbeta)d\sigma(\bx)\\
=&\int_{\Omega}\frac{|\bx-\boy|^{2-m}}{2-m}\big(D_{\bx}f(\bx,\bbeta)\big)d\bx-\int_{\partial\Omega}\frac{|\bx-\boy|^{2-m}}{2-m}n(\bx)f(\bx,\bbeta)d\sigma(\bx).
\end{align*}
It is easy to observe that one can interchange differentiation and integration above if we differentiate $T_kf$, which gives us that
\begin{align}\label{partial0}
&c_{m,k}\frac{\partial}{\partial y_i}T_kf(\boy,\bu)\nonumber\\
=&\int_{\Omega}\frac{y_i-x_i}{|\bx-\boy|^m}\big(D_{\bx}f(\bx,\bbeta)\big)+\frac{|\bx-\boy|^{2-m}}{2-m}\sum_{s=1}^m\frac{\partial \eta_s}{\partial y_i}\frac{\partial}{\partial\eta_s}D_{\bx}f(\bx,\bbeta)d\bx\nonumber\\
&-\int_{\partial\Omega}\frac{y_i-x_i}{|\bx-\boy|^m}n(\bx)f(\bx,\bbeta)+\frac{|\bx-\boy|^{2-m}}{2-m}n(\bx)\sum_{s=1}^m\frac{\partial \eta_s}{\partial y_i}\frac{\partial}{\partial\eta_s}f(\bx,\bbeta)d\sigma(\bx),
\end{align}
where
\begin{align}\label{eqn1}
&\bbeta=\displaystyle\frac{(\bx-\boy)\bu(\bx-\boy)}{|\bx-\boy|^2}=\bu-\frac{\langle\bu,\bx-\boy\rangle(\bx-\boy)}{|\bx-\boy|^2},\nonumber\\
& \eta_s=u_s-\frac{\langle\bu,\bx-\boy\rangle(x_s-y_s)}{|\bx-\boy|^2},\nonumber\\
&\frac{\partial\eta_s}{\partial y_i}=-\frac{2u_i(x_s-y_s)-2\delta_{is}\langle\bx-\boy,\bu\rangle}{|\bx-\boy|^2}+\frac{4\langle\bx-\boy,\bu\rangle(x_s-y_s)(y_i-x_i)}{|\bx-\boy|^4}.
\end{align}
With Stokes'  Theorem, we have
\begin{align*}
&\int_{\partial\Omega_{\epsilon}}\frac{y_i-x_i}{|\bx-\boy|^m}n(\bx)f(\bx,\bbeta)+\frac{|\bx-\boy|^{2-m}}{2-m}n(\bx)\sum_{s=1}^m\frac{\partial \eta_s}{\partial y_i}\frac{\partial f(\bx,\bbeta)}{\partial\eta_s}d\sigma(\bx)\\
=&\int_{\Omega_{\epsilon}}\bigg(\frac{y_i-x_i}{|\bx-\boy|^m}D_{\bx}\bigg)f(\bx,\bbeta)+\bigg(\frac{|\bx-\boy|^{2-m}}{2-m}D_{\bx}\bigg)\sum_{s=1}^m\frac{\partial \eta_s}{\partial y_i}\frac{\partial}{\partial\eta_s}f(\bx,\bbeta)d\bx\\
&+\int_{\Omega_{\epsilon}}\frac{y_i-x_i}{|\bx-\boy|^m}\big(D_{\bx}f(\bx,\bbeta)\big)+\frac{|\bx-\boy|^{2-m}}{2-m}\sum_{s=1}^mD_{\bx}\frac{\partial \eta_s}{\partial y_i}\frac{\partial}{\partial\eta_s}f(\bx,\bbeta)d\bx.
\end{align*}
Plugging into \eqref{partial0}, we have
\begin{align}\label{partial1}
&c_{m,k}\frac{\partial}{\partial y_i}T_kf(\boy,\bu)\nonumber\\
=&\lim_{{\epsilon}\rightarrow 0}\int\limits_{\partial B(\boy,\epsilon)}\frac{x_i-y_i}{|\bx-\boy|^m}n(\bx)f(\bx,\bbeta)-\frac{|\bx-\boy|^{2-m}}{2-m}n(\bx)\sum_{s=1}^m\frac{\partial \eta_s}{\partial y_i}\frac{\partial f(\bx,\bbeta)}{\partial\eta_s}d\sigma(\bx)\nonumber\\
&-\lim_{{\epsilon}\rightarrow 0}\int_{\Omega_{\epsilon}}\bigg(\frac{y_i-x_i}{|\bx-\boy|^m}D_{\bx}\bigg)f(\bx,\bbeta)+\bigg(\frac{|\bx-\boy|^{2-m}}{2-m}D_{\bx}\bigg)\sum_{s=1}^m\frac{\partial \eta_s}{\partial y_i}\frac{\partial f(\bx,\bbeta)}{\partial\eta_s}d\bx\nonumber\\
&+\lim_{{\epsilon}\rightarrow 0}\int_{B(\boy,\epsilon)}\frac{y_i-x_i}{|\bx-\boy|^m}\big(D_{\bx}f(\bx,\bbeta)\big)+\frac{|\bx-\boy|^{2-m}}{2-m}\sum_{s=1}^mD_{\bx}\frac{\partial \eta_s}{\partial y_i}\frac{\partial f(\bx,\bbeta)}{\partial\eta_s}d\bx.
\end{align}
To obtain the mapping property for the operator $T_k$, we notice that the space $C^{1}(\Omega\times\Bm,\Mk)$ is dense in $L^p(\Omega\times\Bm,\Mk)$, this is because $C_c^{\infty}(\Omega\times\Bm,\Mk)$ is dense in $L^p(\Omega\times\Bm,\Mk)$ like in the classical case, and $C_c^{\infty}(\Omega\times\Bm,\Mk)$  is a subset of $C^{1}(\Omega\times\Bm,\Mk)$ in $L^p(\Omega\times\Bm,\Mk)$. Hence, we can prove the mapping property of $T_k$ with the expression of $\frac{\partial}{\partial y_i}T_kf$ given in \eqref{partial1}.
\par
 Firstly, if we let $\bx-\boy=r\bt$ with $r>0, \bt\in\Sm$, then, the homogeneity of $r$ suggests that the last integral in \eqref{partial1} tends to zero when $\epsilon$ goes to zero. Secondly, we notice that
\begin{align*}
&\frac{y_i-x_i}{|\bx-\boy|^m}D_{\bx}=\frac{\partial}{\partial y_i}D_{\bx}\frac{|\bx-\boy|^{2-m}}{2-m}=\frac{\partial}{\partial y_i}\frac{\bx-\boy}{|\bx-\boy|^{m}},\\
&f(\bx,\bbeta)=\int_{\Sm}Z_k(\bbeta,\bov)f(\bx,\bov)dS(\bov).
\end{align*}
Hence, the second integral in \eqref{partial1} can be rewritten as
\begin{align*}
&\lim_{{\epsilon}\rightarrow 0}\int_{\Omega_{\epsilon}}\bigg(\frac{y_i-x_i}{|\bx-\boy|^m}D_{\bx}\bigg)f(\bx,\bbeta)+\bigg(\frac{|\bx-\boy|^{2-m}}{2-m}D_{\bx}\bigg)\sum_{s=1}^m\frac{\partial \eta_s}{\partial y_i}\frac{\partial}{\partial\eta_s}f(\bx,\bbeta)d\bx\\
=&\lim_{{\epsilon}\rightarrow 0}\int_{\Omega_{\epsilon}}\bigg(\partial_{y_i}\frac{|\bx-\boy|^{2-m}}{2-m}D_{\bx}\bigg)f(\bx,\bbeta)+\frac{\bx-\boy}{|\bx-\boy|^{m}}\frac{\partial}{\partial y_i}f(\bx,\bbeta)d\bx\\
=&\lim_{{\epsilon}\rightarrow 0}\int_{\Omega_{\epsilon}}\int_{\Sm}\frac{\partial}{\partial y_i}\bigg[\frac{\bx-\boy}{|\bx-\boy|^{m}}Z_k\bigg(\frac{(\bx-\boy)\bu(\bx-\boy)}{|\bx-\boy|^2},\bov\bigg)\bigg]f(\bx,\bov)dS(\bov)d\bx\\
=&c_{m,k}\lim_{{\epsilon}\rightarrow 0}\int_{\Omega_{\epsilon}}\int_{\Sm}\frac{\partial}{\partial y_i}E_k(\bx-\boy,\bu,\bov)f(\bx,\bov)dS(\bov)d\bx\\
=&c_{m,k}\int_{\Omega}\int_{\Sm}\frac{\partial}{\partial y_i}E_k(\bx-\boy,\bu,\bov)f(\bx,\bov)dS(\bov)d\bx.
\end{align*}
For the first integral in \eqref{partial1}, let $\bx=\boy+\epsilon\bt$, and with \eqref{eqn1}, we have
\begin{align*}
&\lim_{{\epsilon}\rightarrow 0}\int\limits_{\partial B(\boy,\epsilon)}\frac{x_i-y_i}{|\bx-\boy|^m}n(\bx)f(\bx,\bbeta)-\frac{|\bx-\boy|^{2-m}}{2-m}n(\bx)\sum_{s=1}^m\frac{\partial \eta_s}{\partial y_i}\frac{\partial f(\bx,\bbeta)}{\partial\eta_s}d\sigma(\bx)\\
=&\lim_{{\epsilon}\rightarrow 0}\int_{|\bt|=1}t_i\bt f(\boy+\epsilon\bt,\bt\bu\bt)\\
&\quad\quad-\frac{2\bt}{2-m}\bigg[u_i\langle\bt,D_{\bbeta}\rangle+\langle\bt,\bu\rangle\frac{\partial}{\partial\eta_i}+2t_i\langle\bt,\bu\rangle\langle\bt,D_{\bbeta}\rangle\bigg]f(\boy+\epsilon\bt,\bt\bu\bt)d\sigma(\bt)\\
=&\int_{|\bt|=1}t_i\bt f(\boy,\bt\bu\bt)\\
&\quad\quad-\frac{2\bt}{2-m}\bigg[u_i\langle\bt,D_{\bbeta}\rangle+\langle\bt,\bu\rangle\frac{\partial}{\partial\eta_i}+2t_i\langle\bt,\bu\rangle\langle\bt,D_{\bbeta}\rangle\bigg]f(\boy,\bt\bu\bt)d\sigma(\bt).
\end{align*}
Hence, we have 
\begin{align}\label{partial}
&\frac{\partial}{\partial y_i}T_kf(\boy,\bu)=\int_{\Omega}\int_{\Sm}\frac{\partial}{\partial y_i}E_k(\bx-\boy,\bu,\bov)f(\bx,\bov)dS(\bov)d\bx\nonumber\\
&+c_{m,k}^{-1}\int_{|\bt|=1}t_i\bt f(\boy,\bt\bu\bt)d\sigma(\bt)\nonumber\\
&-c_{m,k}^{-1}\int_{|\bt|=1}\frac{2\bt}{2-m}\bigg[u_i\langle\bt,D_{\bbeta}\rangle+\langle\bt,\bu\rangle\frac{\partial}{\partial\eta_i}+2t_i\langle\bt,\bu\rangle\langle\bt,D_{\bbeta}\rangle\bigg]f(\boy,\bt\bu\bt)d\sigma(\bt).
\end{align}
Now, we study the three integrals in \eqref{partial}, respectively. For convenience, we denote the operators
\begin{align*}
&F_1f(\boy,\bu)=\int_{\Omega}\int_{\Sm}\frac{\partial}{\partial y_i}E_k(\bx-\boy,\bu,\bov)f(\bx,\bov)dS(\bov)d\bx,\\
&F_2f(\boy,\bu)=c_{m,k}^{-1}\int_{|\bt|=1}t_i\bt f(\boy,\bt\bu\bt)d\sigma(\bt),\\
&F_3f(\boy,\bu)=-c_{m,k}^{-1}\int_{|\bt|=1}\frac{2\bt}{2-m}\bigg[u_i\langle\bt,D_{\bbeta}\rangle+\langle\bt,\bu\rangle\frac{\partial}{\partial\eta_i}\\
&\quad\quad\quad\quad\quad\quad\quad\quad\quad\quad\quad+2t_i\langle\bt,\bu\rangle\langle\bt,D_{\bbeta}\rangle\bigg]f(\boy,\bt\bu\bt)d\sigma(\bt).
\end{align*}
Hence, we have $\frac{\partial}{\partial y_i}T_k=F_1+F_2+F_3$, and we will show that $\frac{\partial}{\partial y_i}T_kf\in L^p(\Omega\times \Bm,\Mk)$ by proving that $F_jf\in L^p(\Omega\times \Bm,\Mk)$ for $j=1,2,3$, respectively.
\begin{enumerate}
\item Firstly, we consider
\begin{align}\label{p1}
&|F_1f(\boy,\bu)|\leq\int_{\Omega}\int_{\Sm}\bv\frac{\partial}{\partial y_i}E_k(\bx-\boy,\bu,\bov)f(\bx,\bov)\bv dS(\bov)d\bx\nonumber\\
=&\int_{\Omega}\int_{\Sm}\bv\frac{\partial}{\partial y_i}\bigg[\frac{\bx-\boy}{|\bx-\boy|^{m}}Z_k\bigg(\frac{(\bx-\boy)\bu(\bx-\boy)}{|\bx-\boy|^2},\bov\bigg)\bigg]f(\bx,\bov)\bv dS(\bov)d\bx\nonumber\\
=&\int_{\Omega}\int_{\Sm}\bv\bigg[\bigg(\frac{\partial}{\partial y_i}\frac{\bx-\boy}{|\bx-\boy|^{m}}\bigg)Z_k\bigg(\frac{(\bx-\boy)\bu(\bx-\boy)}{|\bx-\boy|^2},\bov\bigg)\nonumber\\
&+\frac{\bx-\boy}{|\bx-\boy|^{m}}\bigg(\frac{\partial}{\partial y_i}Z_k\bigg(\frac{(\bx-\boy)\bu(\bx-\boy)}{|\bx-\boy|^2},\bov\bigg)\bigg]f(\bx,\bov)\bv dS(\bov)d\bx\nonumber\\
=&\int_{\Omega}\int_{\Sm}\bv\bigg[\bigg(\frac{\partial}{\partial y_i}\frac{\bx-\boy}{|\bx-\boy|^{m}}\bigg)Z_k\bigg(\frac{(\bx-\boy)\bu(\bx-\boy)}{|\bx-\boy|^2},\bov\bigg)\nonumber\\
&+\frac{\bx-\boy}{|\bx-\boy|^{m}}\sum_{s=1}^m\frac{\partial\eta_s}{\partial y_i}\frac{\partial}{\partial \eta_s}Z_k(\bbeta,\bov)\bigg]f(\bx,\bov)\bv dS(\bov)d\bx.
\end{align}
Notice that $Z_k(\bbeta,\bov)$ and $\frac{\partial}{\partial{\eta_s}}Z_k(\bbeta,\bov)$ are all bounded for $\bov\in\Sm,\bu\in \Bm$. Further, with the three equations in \eqref{eqn1}, one can see that the homogeneity of $|\bx-\boy|$ is $-m$. Hence, we obtain the equation \eqref{p1}
\begin{align}\label{p2}
\leq C\int_{\Omega}\frac{1}{|\bx-\boy|^{m}}\bigg[\int_{\Sm}|f(\bx,\bov)|dS(\bov)\bigg]d\bx.
\end{align}
Since $f\in L^p(\Omega\times \Bm,\Mk)$, we can easily see that
\begin{align*}
\int_{\Omega}\bv\int_{\Sm}|f(\bx,\bov)|dS(\bov)\bv^pd\bx&\leq \int_{\Omega}\int_{\Sm}|f(\bx,\bov)|^pdS(\bov)pd\bx\\
&= ||f||^p_{L^p(\Omega\times \Bm,\Mk)}<\infty,
\end{align*}
in other words,
\begin{align*}
\Phi(\bx):=\int_{\Sm}|f(\bx,\bov)|dS(\bov)\in L^p(\Omega,\Clm).
\end{align*}
Therefore, with Theorem of Calderon and Zygmund \cite[Theorem 3.1, XI]{MP} and \eqref{p2}, we obtain that
\begin{align*}
&||F_1f(\boy,\bu)||^p_{L^p(\Omega\times \Bm,\Mk)}\leq C\int_{\Omega}\bv\int_{\Omega}\frac{1}{|\bx-\boy|^m}\Phi(\bx)d\bx\bv^pd\boy\\
\leq& C\int_{\Omega}|\Phi(\bx)|^pd\bx=C\int_{\Omega}\bigg(\int_{\Sm}|f(\bx,\bov)|dS(\bov)\bigg)^pd\bx\\
\leq& C\int_{\Omega}\int_{\Sm}|f(\bx,\bov)|^pdS(\bov)d\bx= C||f||_{L^p(\Omega\times \Bm,\Mk)}<\infty.
\end{align*}
\item Now, let $\bbeta=\bt\bu\bt$, we have
\begin{align*}
||F_2f||^p_{L^p(\Omega\times \Bm,\Mk)}=&\int_{\Omega}\int_{\Sm}\bv c_{m,k}^{-1}\int_{|\bt|=1}t_i\bt f(\boy,\bt\bu\bt)d\sigma(\bt)\bv^p dS(\bu)d\boy\\
=& c_{m,k}^{-p}\int_{\Omega}\int_{\Sm}\bv \int_{|\bt|=1}t_i\bt f(\boy,\bbeta)d\sigma(\bt)\bv^p dS(\bbeta)d\boy\\
=&c_{m,k}^{-p}\int_{\Omega}\int_{\Sm}|f(\boy,\bbeta)|^p\bv \int_{|\bt|=1}t_i\bt d\sigma(\bt)\bv^p dS(\bbeta)d\boy\\
\leq&C\int_{\Omega}\int_{\Sm}|f(\boy,\bbeta)|^pdS(\bbeta) d\boy\\
=&C||f||_{L^p(\Omega\times \Bm,\Mk)}<\infty.
\end{align*}
\item It is easy to observe that
\begin{align}\label{eqn3}
&||F_3f||^p_{L^p(\Omega\times \Bm,\Mk)}\nonumber\\
\leq &C_1\bv\bv\int_{|\bt|=1}u_i\bt\langle\bt,D_{\bbeta}\rangle f(\boy,\bt\bu\bt)d\sigma(\bt)\bv\bv^p_{L^p(\Omega\times \Bm,\Mk)}\nonumber\\
&+C_2\bv\bv\int_{|\bt|=1}\bt\langle\bt,\bu\rangle\frac{\partial}{\partial\eta_i}f(\boy,\bt\bu\bt)d\sigma(\bt)\bv\bv^p_{L^p(\Omega\times \Bm,\Mk)}\nonumber\\
&+C_3\bv\bv\int_{|\bt|=1} t_i\bt\langle\bt,\bu\rangle\langle\bt,D_{\bbeta}\rangle f(\boy,\bt\bu\bt)d\sigma(\bt)\bv\bv^p_{L^p(\Omega\times \Bm,\Mk)}.
\end{align}
Let $\bbeta=\bt\bu\bt$, then we have $\eta_i=u_i-2\langle \bt,\bbeta\rangle t_i$.
Hence, the first term in \eqref{eqn3} becomes
\begin{align}\label{imp}
&\bv\bv\int_{|\bt|=1}u_i\bt\langle\bt,D_{\bbeta}\rangle f(\boy,\bt\bu\bt)d\sigma(\bt)\bv\bv^p_{L^p(\Omega\times \Bm,\Mk)}\nonumber\\
=&\int_{\Omega}\int_{\Sm}\bv\int_{|\bt|=1} u_i\bt\langle\bt,D_{\bbeta}\rangle f(\boy,\bt\bu\bt)d\sigma(\bt)\bv^pdS(\bu)d\boy\nonumber\\
=&\int_{\Omega}\int_{\Sm}\bv\int_{|\bt|=1}(\eta_i+2\langle \bt,\bbeta\rangle t_i) \bt\langle\bt,D_{\bbeta}\rangle f(\boy,\bbeta)d\sigma(\bt)\bv^pdS(\bbeta)d\boy\nonumber\\
\leq&C\int_{\Omega}\int_{\Sm}\bv\sum_{j=1}^m\frac{\partial f(\boy,\bbeta)}{\partial\eta_j}\bv^pdS(\bbeta)d\boy\nonumber\\
\leq& C\sum_{j=1}^m\int_{\Omega}\int_{\Sm}\bv\frac{\partial f(\boy,\bbeta)}{\partial\eta_j}\bv^pdS(\bbeta)d\boy,
\end{align}
where the first inequality above comes from the fact that $\bt,\bu\in\Sm$, which is obviously bounded. Further, we notice that the function $f(\boy,\bbeta)$ is a $k$-homogeneous harmonic polynomial with respect to $\bbeta$, therefore, equation \eqref{imp} becomes
\begin{align*}
\leq& c_m\int_{\Omega}\int_{B(0,\frac{1}{2})}|\nabla_{\bbeta} f(\boy,\bbeta)|^pd\bbeta d\boy\\
\leq&c_m\int_{\Omega}\int_{\Bm}|f(\boy,\bbeta)|^pd\bbeta d\boy=c_m'\int_{\Omega}\int_{\Sm}|f(\boy,\bbeta)|^pdS(\bbeta)d\boy\\
=&c_m'||f||_{L^p(\Omega\times\Bm,\Mk)},
\end{align*}
where $c_m,c_m'$ are positive constants only depending on $m$, and the second inequality above comes from the interior gradient estimate for harmonic functions. 
\par
We can also have similar results for the other two terms in \eqref{eqn3}, in other words,
\begin{align*}
&\bv\bv\int_{|\bt|=1}\bt\langle\bt,\bu\rangle\frac{\partial}{\partial\eta_i}f(\boy,\bt\bu\bt)d\sigma(\bt)\bv\bv^p_{L^p(\Omega\times \Bm,\Mk)}\leq C||f||_{L^p(\Omega\times\Bm,\Mk)}\\
&\bv\bv\int_{|\bt|=1} t_i\bt\langle\bt,\bu\rangle\langle\bt,D_{\bbeta}\rangle f(\boy,\bt\bu\bt)d\sigma(\bt)\bv\bv^p_{L^p(\Omega\times \Bm,\Mk)}\leq C||f||_{L^p(\Omega\times\Bm,\Mk)}.
\end{align*}
Therefore, we have 
\begin{align*}
||F_3f||_{L^p(\Omega\times \Bm,\Mk)}\leq C||f||_{L^p(\Omega\times\Bm,\Mk)},
\end{align*}
which leads to the fact that $\frac{\partial}{\partial y_i}T_k=F_1+F_2+F_3\in L^p(\Omega\times \Bm,\Mk)$. Combining with Theorem \ref{Tkb}, we have $T_kf\in W^{1}_p(\Omega\times\Bm,\Mk)$ for all $f\in L^p(\Omega\times\Bm,\Mk)$, which completes the proof.
\end{enumerate}
\end{proof}
%%%%%%%%%%%%%%%%%%%%%%%%%%%%%%%
%%%%%%%%%%%%%%%%%%%%%%%%%%%%%%%
\section{A Hodge decomposition in higher spin theory}

%%%%%%%%%%%

We also have a point-wise estimate by an $L^2$ norm.
\begin{proposition}\label{pestimate}
Let $f\in L^2(\Omega\times \Bm,\Mk)$ and $D\subset\Omega$ is compact. Then, we have
\begin{align*}
\sup_{\boy\in D,\bu\in\Sm}|f(\boy,\bu)|\leq C_{k,D}||f||_{L^2(\Omega\times \Bm,\Mk)}.
\end{align*}
\end{proposition}
\begin{proof}
Let $\boy\in D$ be an arbitrary point and $B(\boy,r)\subset D$ for a sufficiently small $r>0$. On one hand, we have
\begin{align}\label{eqn1}
&||f||^2_{L^2(\Omega\times \Bm)}=\int_{\Omega}\int_{\mathbb{S}^{m-1}}|f(\bx,\bov)|^2dS(\bov)d\bx\nonumber\\
\geq&\int_{B(\boy,r)}\int_{\Sm}|f(\bx,\bov)|^2dS(\bov)d\bx\nonumber\\
=&\int_{B(\boy,r)}\int_{\Sm}|f(\bx,\frac{(\bx-\boy)\bu\widetilde{(\bx-\boy)}}{|\bx-\boy|^2})|^2dS(\bu)d\bx\nonumber\\
=&\int_{B(\boy,r)}\int_{\Sm}\bigg[\overline{f(\bx,\frac{(\bx-\boy)\bu\widetilde{(\bx-\boy)}}{|\bx-\boy|^2})}\cdot f(\bx,\frac{(\bx-\boy)\bu\widetilde{(\bx-\boy)}}{|\bx-\boy|^2})\bigg]_0dS(\bu)d\bx.
\end{align}
Note that $f(\bx,\bu)$ is monogenic with respect to $\bu$, which implies it is also harmonic in $\bu$. Since harmonicity is invariant under rotation, $f\bigg(\bx,\frac{(\bx-\boy)\bu\widetilde{(\bx-\boy)}}{|\bx-\boy|^2}\bigg)$ is also harmonic in $\bu$. Therefore, with the Fischer decomposition, we have
\begin{align*}
f\bigg(\bx,\frac{(\bx-\boy)\bu\widetilde{(\bx-\boy)}}{|\bx-\boy|^2}\bigg)=f_1(\bx,\bu)+uf_2(\bx,\bu)
\end{align*}
where
\begin{align*}
&f_1(\bx,\bu)=P_k^+f(\bx,\frac{(\bx-\boy)\bu\widetilde{(\bx-\boy)}}{|\bx-\boy|^2})\in \Mk(\bu),\\
 &\bu f_2(\bx,\bu)=P_k^-f(\bx,\frac{(\bx-\boy)\bu\widetilde{(\bx-\boy)}}{|\bx-\boy|^2})\in\bu\Mkk(\bu).
\end{align*}
The Cauchy's theorem for monogenic functions \cite[Corollary 1]{DJR} implies that
\begin{align*}
\int_{\Sm}\overline{f_1(\bx,\bu)}\bu f_2(\bx,\bu)dS(\bu)=0.
\end{align*}
Therefore, \eqref{eqn1} is equal to
\begin{align*}
=&\int_{B(\boy,r)}\int_{\Sm}\big[\overline{(f_1(\bx,\bu)+uf_2(\bx,\bu))}(f_1(\bx,\bu)+uf_2(\bx,\bu))\big]_0dS(\bu)d\bx\\
=&\int_{B(\boy,r)}\int_{\Sm}\big[\overline{(f_1(\bx,\bu)}f_1(\bx,\bu)\big]_0dS(\bu)d\bx\\
&+\int_{B(\boy,r)}\int_{\Sm}\big[\overline{(f_2(\bx,\bu)}f_2(\bx,\bu)\big]_0dS(\bu)d\bx\\
\geq&\int_{B(\boy,r)}\int_{\Sm}\big[\overline{(f_1(\bx,\bu)}f_1(\bx,\bu)\big]_0dS(\bu)d\bx\\
=&\int_{B(\boy,r)}\int_{\Sm}\bv P^+_{k,u}f\bigg(\bx,\frac{(\bx-\boy)\bu\widetilde{(\bx-\boy)}}{|\bx-\boy|^2}\bigg)\bv^2dS(\bu)d\bx\\
\geq&C_r\int_{\Sm}\bv \int_{B(\boy,r)}P^+_{k,u}f\bigg(\bx,\frac{(\bx-\boy)\bu\widetilde{(\bx-\boy)}}{|\bx-\boy|^2}\bigg)d\bx\bv^2 dS(\bu)\\
=&C_{k,r}\int_{\Sm}|f(\boy,\bu)|^2dS(\bu)\geq C_{k,r}\sup_{\bu\in\Sm}|f(\boy,\bu)|^2,
\end{align*}
where $C_{k,r}$ stands for a positive constant depending on $k$ and $r$ and the last equation comes from the mean-value property given in Corollary \ref{MVP1}. Since $\boy\in D$ is arbitrary, we have that
\begin{align*}
\sup_{\boy\in D,\bu\in\Sm}|f(\boy,\bu)|\leq C_{k,D}||f||_{L^2(\Omega\times \Bm)}.
\end{align*}
where $D\subset \Omega$ is an arbitrary compact set.
\end{proof}
Now, we claim that 
\begin{theorem}
The set $\ker R_k\cap L^2(\Omega\times \Bm,\Mk)$ is a closed subspace in $L^2(\Omega\times \Bm,\Mk)$.
\end{theorem}
\begin{proof}
Let $\{f_n\}$ be a Cauchy sequence in $\ker R_k\cap L^2(\Omega\times \Bm,\Mk)$, and it converges to a function $f\in L^2(\Omega\times \Bm,\Mk)$. Proposition \ref{pestimate} tells us that there exists a function 
$$g: \Omega\times \Bm\longrightarrow\Clm$$
defined by $g(\boy,\bu)=\underset{n\rightarrow \infty}{\lim}f_n(\boy,\bu)$. Further, $\{f_n\}$ converges uniformly to $g$ on every $D\times\Sm$, where $D\subset \Omega$ is compact. This implies that one can interchange limit and differentiation on the sequence $\{f_n\}$, which gives rise to 
\begin{align*}
&D_{\bu}g(\boy,\bu)=\lim_{n\rightarrow\infty}D_{\bu}f_n(\boy,\bu)=0,\\
& R_kg(\boy,\bu)=\lim_{n\rightarrow\infty}R_kf_n(\boy,\bu)=0,
\end{align*}
for $(\boy,\bu)\in\Omega\times \Bm$. In other words, $g\in \ker R_k\cap \Mk(\bu)$. Now, for any compact set $D\subset\Omega$, we have
\begin{align*}
0\leq& \int_{D}\int_{\Sm}|f(\boy,\bu)-g(\boy,\bu)|^2dS(\bu)d\boy\\
\leq& \int_{D}\int_{\Sm}|f(\boy,\bu)-f_n(\boy,\bu)|^2dS(\bu)d\boy\\
&+\int_{D}\int_{\Sm}|f_n(\boy,\bu)-g(\boy,\bu)|^2dS(\bu)d\boy\\
=&||f-f_n||_{L^2(\Omega\times\Bm)}+\int_{D}\int_{\Sm}|f_n(\boy,\bu)-g(\boy,\bu)|^2dS(\bu)d\boy,
\end{align*}
which tends to zero when $n\rightarrow\infty$ above. Therefore, we have $f=g\in R_k\cap L^2(\Omega\times \Bm,\Mk)$, which completes the proof.
\end{proof}
Now, we denote 
\begin{align*}
&F_{\partial\Omega}f(\by,\bu):=\int_{\partial\Omega}\int_{\Sm}E_k(\bx-\boy,\bu,\bov)n(\bx)f(\bx,\bov)dS(\bov)d\bx,\ \boy\in\Omega,\\
&S_{\partial\Omega}f(\bo{\omega},\bu):=p.v.\int_{\partial\Omega}\int_{\Sm}E_k(\bx-\bo{\omega},\bu,\bov)n(\bx)f(\bx,\bov)dS(\bov)d\bx,\ \bo{\omega}\in\partial\Omega.
\end{align*}
Then, the Plemelj formula for Rarita-Schwinger type operators are stated as follows. It is worth pointing out that the Plemelj formula given in \cite{Li} is on Lipschitz domains. For our own convenience, here we state it slightly differently.
\begin{theorem}\cite[Theorem 1]{Li}\label{Plemelj} 
Let $\Omega\subset\Rm$ be a domain with smooth boundary $\partial\Omega$. $\boy(t)$ is a smooth path in $\Rm$ and it has non-tangential limit $\bo{\omega}\in\partial\Omega$ as $t\rightarrow 0$. Then, for each H\"older continuous function $f:\Omega\times\Bm\longrightarrow\Clm$ defined on $\Omega$, we have 
	\begin{align*}
		\lim_{t\to0}F_{\partial\Omega}f(\boy(t),\bu)
		=\begin{cases}
			\dfrac{f(\bo{\omega},\bu)}{2}+S_{\partial\Omega}f(\bo{\omega},\bu),(\boy(t),\bu)\in\Omega\times\Bm;\\
			-\dfrac{f(\bo{\omega},\bu)}{2}+S_{\partial\Omega}f(\bo{\omega},\bu),(\boy(t),\bu)\in\big(\mathbb{R}^m\backslash\overline{\Omega}\big)\times\Bm.
\end{cases}
	\end{align*}
	\end{theorem}
	With Theorem \ref{Plemelj} above, we can have a result on fermionic continuation as follows.
	\begin{corollary}\label{extension}
	Let $\Omega\in\Rm$ be a bounded domain with smooth boundary $\partial\Omega$. On one hand, a function $f$ represents the boundary values of a fermionic function $S_{\partial\Omega}f$ defined in $\Omega\times\Bm$ if and only if
\begin{align*}	
S_{\partial\Omega}f(\boy,\bu)=\frac{f(\boy,\bu)}{2}\ \text{for all}\ \boy\in\partial\Omega, \bu\in\Bm.
\end{align*}
 On the other hand, a function $f$ represents the boundary values of a fermionic function $S_{\partial\Omega}f$ defined in $\Rm\times\Bm\backslash{\Omega\times\Bm}$ if and only if \begin{align*}	
S_{\partial\Omega}f(\boy,\bu)=-\frac{f(\boy,\bu)}{2}\ \text{for all}\ \boy\in\partial\Omega, \bu\in\Bm.
\end{align*}
	\end{corollary}
	\begin{proof}
	Let $g$ be the fermionic continuation into the domain $\Omega\times\Bm$ of the function $f$ given on $\partial\Omega\times\Bm$. Then, the Cauchy integral formula shows us that $g(\boy,\bu)=F_{\partial\Omega}f$. Hence, the non-tangential boundary values of $g$ are $f$. Further, with the Plemelj formula in Theorem \ref{Plemelj}, we have
	\begin{align*}
	f(\boy,\bu)=\frac{f(\boy,\bu)}{2}+S_{\partial\Omega}f(\boy,\bu),\ \text{for\ all}\ \boy\in\partial\Omega,\bu\in\Bm,
	\end{align*} 
	which gives us that
	\begin{align*}
	S_{\partial\Omega}f(\boy,\bu)=\frac{f(\boy,\bu)}{2},\ \text{for\ all}\ \boy\in\partial\Omega,\bu\in\Bm.
	\end{align*} 
	If vice, versa, we have
	\begin{align*}
	S_{\partial\Omega}f(\boy,\bu)=\frac{f(\boy,\bu)}{2},\ \text{for\ all}\ \boy\in\partial\Omega,\bu\in\Bm.
	\end{align*}
	Then, Theorem \ref{Plemelj} gives us that $F_{\partial\Omega}f$ has the boundary value $f$. Hence, it is the fermionic continuation of $f$ into $\Omega\times\Bm$. The case of exterior domain can be obtained with a similar argument as above.
	\end{proof}
	We denote $\mathcal{A}^p(\Omega\times\Bm,\Mk):=\ker R_k\cap L^p(\Omega\times\Bm,\Mk)$, and we introduce a Hodge decomposition of $L^p(\Omega\times\Bm,\Mk)$ with $1<p<\infty$ as follows.
\begin{theorem}[Hodge decomposition]
Let $\Omega\subset\Rm$ be a bounded domain with smooth boundary $\partial\Omega$ and $1<p<\infty$. Then, the space $L^p(\Omega\times\Bm,\Mk)$ allows the following orthogonal decomposition 
\begin{align*}
L^p(\Omega\times\Bm,\Mk)=\mathcal{A}^p(\Omega\times\Bm,\Mk)\bigoplus \big(R_k\overset{0}{W_{p}^{1}}(\Omega\times\Bm,\Mk)\big)
\end{align*}
with respect to the inner product
\begin{align*}
\langle f,g\rangle:=\int_{\Omega}\int_{\Sm}\overline{f(\bx,\bov)}g(\bx,\bov)dS(\bov)d\bx,
\end{align*}
where $\overset{0}{W_{p}^{1}}(\Omega\times\Bm,\Mk)$ is the subspace of functions in $W_{p}^{1}(\Omega\times\Bm,\Mk)$ with zero values on $\partial\Omega\times\Bm$.
\end{theorem}  
\begin{proof}
Let $Y:=L^p(\Omega\times\Bm,\Mk)\ominus\mathcal{A}^p(\Omega\times\Bm,\Mk)$ be the orthogonal complement to $\mathcal{A}^p(\Omega\times\Bm,\Mk)$ with respect to the inner product given above. For any function $f\in Y$, then, Theorem \ref{Tkb} tells us that $g=T_kf\in L^p(\Omega\times\Bm,\Mk)$. Since $R_k$ is the left inverse of $T_k$, we have $f=R_k g$. For any $\varphi\in \mathcal{A}^p(\Omega\times\Bm,\Mk)$, which is orthogonal to elements in $Y$ as
\begin{align*}
0=&\int_{\Omega}\int_{\Sm}\overline{\varphi(\bx,\bov)}f(\bx,\bov)dS(\bov)d\bx\\
=&\int_{\Omega}\int_{\Sm}\overline{\varphi(\bx,\bov)}(R_kg)(\bx,\bov)dS(\bov)d\bx.
\end{align*}
In particular, let $\varphi(\bx,\bov)=\overline{E_k(\bx_l-\boy,\bu,\bov)}$, where $\{\bx_l\}$ is dense in $\Rm\backslash\overline{\Omega}$. Then, we have
\begin{align*}
0=&\int_{\Omega}\int_{\Sm}E_k(\bx_l-\boy,\bu,\bov)(R_kg)(\bx,\bov)dS(\bov)d\bx\\
=&\int_{\partial\Omega}\int_{\Sm}E_k(\bx_l-\boy,\bu,\bov)n(\bx)g(\bx,\bov)dS(\bov)d\sigma(\bx)\\
&-\int_{\Omega}\int_{\Sm}\bigg(E_k(\bx_l-\boy,\bu,\bov)R_k\bigg)g(\bx,\bov)dS(\bov)d\bx\\
=&\int_{\partial\Omega}\int_{\Sm}E_k(\bx_l-\boy,\bu,\bov)n(\bx)g(\bx,\bov)dS(\bov)d\sigma(\bx)=F_{\partial\Omega}(\text{tr}g)(\bx,\bu),
\end{align*}
where $\text{tr}g$ stands for the trace of $g$. Hence, we have $F_{\partial\Omega}(\text{tr}g)=0$ on $\Rm\times\Bm\backslash\overline{\Omega}\times\Bm$ for continuation. Then, Theorem \ref{Plemelj} gives us that 
\begin{align*}
S_{\partial\Omega}f(\bo{\omega},\bu)=\frac{f(\bo{\omega},\bu)}{2}.
\end{align*}
Therefore, we know that $\text{tr}g$ can be extended fermionicly into $\Omega\times\Bm$ by Corollary \ref{extension}, and we denote its extension by $h$. Then, we have $\text{tr}_{\partial\Omega}g=\text{tr}_{\partial\Omega}h$. Now, we denote ${\gamma}=g-h$, and we obviously have $\text{tr}_{\partial\Omega}{\gamma}=0$ and ${\gamma}\in W_{p}^{1}(\Omega\times\Bm,\Mk)$. Further, we can see that $R_k\gamma=R_kg=f\in Y$, which completes the proof.
\end{proof}
The Hodge decomposition above gives rise to two orthogonal projections as follows.
\begin{align*}
&P:\ L^p(\Omega\times\Bm,\Mk)\longrightarrow \mathcal{A}^p(\Omega\times\Bm,\Mk),\\
&Q:\ L^p(\Omega\times\Bm,\Mk)\longrightarrow\big(R_kW_{p}^{1}(\Omega\times\Bm,\Mk)\big).
\end{align*}
In particular, the projection $P$ can be considered as a generalization of the classical Bergman projection. Hence, we call the space $\mathcal{A}^p(\Omega\times\Bm,\Mk)$ by \emph{fermionic Bergman spaces}. In particular, we can see that $\mathcal{A}^2(\Omega\times\Bm,\Mk)$ is a right Hilbert Clifford module with respect to the inner product
\begin{align*}
\langle f,g\rangle:=\int_{\Omega}\int_{\Sm}\overline{f(\bx,\bov)}g(\bx,\bov)dS(\bov)d\bx.
\end{align*}
Further, with Proposition \ref{pestimate} and Aronszajn-Bergman Theorem in \cite[Theorem 1.3]{Bra}, there is a reproducing kernel for $\mathcal{A}^2(\Omega\times\Bm,\Mk)$. Hence, this gives rise to a theory of generalized Bergman spaces in higher spin cases, which will be investigated in an upcoming article.
% This environment is for acknowledgements
%%%%%%%%%%%%%%%%%%%%%%%%
% ------------------------------------------------------------------------
\subsection*{Acknowledgements}
Chao Ding is supported by the National Natural Science Foundation
of China (No. 12271001) and the Natural Science Foundation of Anhui Province (No. 2308085MA03). 
\subsection*{Disclosure statement}
No potential conflict of interest was reported by the author.

% ------------------------------------------------------------------------
\end{document}